\documentclass[11pt]{article}
\usepackage[psamsfonts]{amsfonts}
\usepackage{amsmath,amssymb,amsthm,dsfont}
\usepackage{graphicx}               
\usepackage[usenames]{color}

\usepackage[textsize=small]{todonotes}       

\newcommand{\ch}[1]{#1}

\newtheorem{theorem}{Theorem}[section]

\newtheorem{definition}[theorem]{Definition}
\newtheorem{lemma}[theorem]{Lemma}
\newtheorem{corollary}[theorem]{Corollary}

\newtheorem{proposition}[theorem]{Proposition}

\newtheorem*{rem}{Remark}

\newcommand{\eq}{\begin{equation}}
\newcommand{\en}{\end{equation}}
\newcommand{\nn}{\nonumber}

\newcommand{\prob}{\mathbb P}
\newcommand{\expec}{\mathbb E}
\newcommand{\ind}{\mathds 1}

\newcommand{\atanh}{{\rm atanh}}
\newcommand{\dint}{{\rm d}}
\newcommand{\bbR}{\mathbb{R}}

\newcommand{\calT}{\mathcal{T}}
\newcommand{\sss}{\scriptscriptstyle}

\numberwithin{equation}{section}

\title  {Ising models on power-law random graphs}

\author
{
Sander Dommers
\footnote{
  Eindhoven University of Technology, Department of Mathematics and Computer Science,
  P.O.\ Box 513, 5600 MB Eindhoven, The Netherlands.
  E-mail: {\tt s.dommers@tue.nl, rhofstad@win.tue.nl}
}
\and Cristian Giardin\`a
\footnote{
  Modena and Reggio Emilia University, via Allegri 9, 42100 Reggio Emilia, Italy. E-mail: {\tt cristian.giardina@unimore.it}
}
\and Remco van der Hofstad$\ ^*$
}

\date{\today}

\begin{document}

\maketitle

\begin{abstract}
We study a ferromagnetic Ising model on random graphs with a power-law degree distribution and compute the thermodynamic limit of the pressure when the mean degree is finite (degree exponent $\tau>2$), for which the random graph has a tree-like structure. For this, we \ch{closely follow} the analysis by Dembo and Montanari \ch{in} \cite{DemMon10} which assumes finite variance degrees ($\tau>3$)\ch{, adapting it when necessary and also simplifying it when possible}. \ch{Our results also apply in cases where the degree distribution does not obey a power law.}

We further identify the thermodynamic limits of various physical quantities, such as the magnetization and the internal energy.
\end{abstract}

\section{Introduction and results}
\label{sec-intro}
In this article we study the behavior of the Ising model on complex networks. There are many real-world examples of complex networks. In \cite{New03}, Newman divided such networks into four categories: social, information, technological and biological networks. There has been much interest in the functionality of such networks in recent years \cite{AlbBar02,New03,Str01}. The Ising model is a paradigm model in statistical physics for cooperative behavior \cite{Nis05,Nis09}.

De Sanctis and Guerra studied this model on Erd{\H o}s-R\'enyi random graphs in the high and zero temperature regime \cite{SanGue08}. In \cite{DemMon10}, Dembo and Montanari study a ferromagnetic Ising model on locally tree-like graphs, where they assume that the degree distribution of the graph has finite variance. The Ising model on the $k$-regular graph where there is no external magnetic field is studied in more detail in~\cite{MonMosSly09}. \ch{In this paper, the Gibbs measures are studied and it is proved that they converge to a symmetric linear combination of the plus and the minus Gibbs measure, while other Gibbs measures (of which there are uncountably many) are not seen.}

Many real-life networks are reported to have an infinite variance degree distribution (see e.g. \cite{New03}) and, therefore, it is interesting to generalize the analysis of the Ising model on random graphs to this setting. In this article we shall extend and simplify the analysis in \cite{DemMon10} to the case where the variances of the degrees are infinite, but their means remain finite. In particular, we shall prove that the explicit expression for the pressure found in~\cite{DemMon10} remains valid in the case of infinite variance degrees.

This research fits into a general effort to study the relation of the topology of networks and the behavior of processes on them. An overview of results by physicists can for example be found in \cite{DorGolMen08}. Also mathematically rigorous results for processes on power-law random graphs were published recently, for example for the contact process \cite{ChaDur09} and first passage percolation \cite{BhaHofHoo10}.

In this section we will first define the model and then state our main results. Furthermore we will discuss these results and give an overview of the proof. The remainder of the proof can be found in the subsequent sections.

\subsection{Model definition}\label{sec-infmean}
We start by defining Ising models on finite graphs. Consider a random graph sequence $\{G_n\}_{n \geq 1}$, where $G_n=(V_n,E_n)$, with vertex set $V_n=[n] \equiv \{1,\ldots,n\}$ and some random edge set $E_n$. To each vertex $i\in [n]$ we assign an Ising spin $\sigma_i = \pm 1$. A configuration of spins is denoted by $\sigma=\{\sigma_i : i\in [n]\}$. The {\em Ising model on $G_n$} is then defined by the Boltzmann distribution \eq\label{eq-boltzmann}
\mu_n(\sigma) = \frac{1}{Z_n(\beta, \underline{B})} \exp \left\{\beta \sum_{(i,j) \in E_n} \sigma_i \sigma_j + \sum_{i \in [n]} B_i \sigma_i\right\}.
\en
Here, $\beta \geq 0$ is the inverse temperature and $\underline{B}=\{B_i : i \in [n]\} \in \bbR^n$ is the vector of external magnetic fields. We will write $B$ instead of $\underline{B}$ for a uniform external field, i.e., $B_i=B$ for all $i\in[n]$. The partition function $Z_n(\beta,\underline{B})$ is \ch{the} normalization factor:
\eq
Z_n(\beta,\underline{B}) = \sum_{\sigma \in \{-1,+1\}^n} \exp \left\{\beta \sum_{(i,j) \in E_n} \sigma_i \sigma_j + \sum_{i \in [n]} B_i \sigma_i\right\}.
\en
We let $\big<\cdot\big>_\mu$ denote the expectation with respect to the Ising measure $\mu$, i.e., for every bounded function $f: \{-1,+1\}^n \rightarrow \bbR$,
\eq
\big<f(\sigma)\big>_{\mu_n} =  \sum_{\sigma \in \{-1,+1\}^n} f(\sigma) \mu_n(\sigma).
\en
The main quantity we shall study is the {\em pressure} per particle, which is defined as
\eq
\psi_n(\beta, B) = \frac{1}{n} \log Z_n(\beta, B),
\en
in the thermodynamic limit of $n \rightarrow \infty$.

We shall assume that the graph sequence $\{G_n\}_{n\geq1}$ is {\em locally like a homogeneous random tree}, {\em uniformly sparse} and has a {\em degree distribution} with {\em strongly finite mean}. We make these assumptions precise below, but we shall first introduce some notation.

For a probability distribution over the non-negative integers $P=\{P_k : k\geq 0\}$ we define its {\em size-biased law} $\rho=\{\rho_k : k\geq 0\}$ by
\eq \label{eq-defrho}
\rho_k = \frac{(k+1) P_{k+1}}{\overline{P}},
\en
where $\overline{P}=\sum_{k \geq 0} k P_k$ is the expected value of $P$. Similarly, we write $\overline{\rho}=\sum_{k \geq 0} k \rho_k$ for the expected value of $\rho$. The random rooted tree $\calT(P,\rho,\ell)$ is a branching process with $\ell$ generations, where the root offspring has distribution $P$ and the vertices in each next generation have offsprings that are {\em independent and identically distributed} (i.i.d.) with distribution $\rho$. \ch{We write $\prob$ for the law of $\calT(P,\rho,\infty)$ and write} $\calT(\rho,\ell)$ when the offspring at the root also has \ch{distribution} $\rho$.

We write that an event $\mathcal{A}$ holds \emph{almost surely} (a.s.) if $\prob[\mathcal{A}]=1$.
The ball of radius $r$ around vertex $i$, $B_i(r)$, is defined as the graph induced by the vertices at graph distance at most $r$ from vertex $i$.
For two rooted trees $\calT_1$ and $\calT_2$, we write that $\calT_1 \simeq \calT_2$, when \ch{there exists a bijective map from the vertices of $\calT_1$ to those of $\calT_2$ that preserves the adjacency relations.}

\begin{definition}[Local convergence to homogeneous trees]\label{ass-convtree} Let $\prob_n$ denote the law induced on the ball $B_i(t)$ in $G_n$ centered at a uniformly chosen vertex $i\in[n]$. We say that the graph sequence $\{G_n\}_{n\geq 1}$ is {\em locally tree-like} with asymptotic degree distribution $P$ when, for any rooted tree $\calT$ with $t$ generations, we have that, a.s.,
\eq
\lim_{n\rightarrow\infty} \prob_n [B_i(t) \simeq \calT] = \prob[\calT(P,\rho,t) \simeq \calT].
\en
\end{definition}
Note that this implies that the degree of a uniformly chosen vertex of the graph has asymptotic law $P$. In \cite{DemMon10}, it is assumed that the asymptotic degree distribution $P$ has finite variance. This is not a necessary condition and we shall prove that it is sufficient to assume that the degree distribution has a finite $(1+\varepsilon)$-th moment for some $\varepsilon>0$:
\begin{definition}[Strongly finite mean degree distribution]\label{ass-degdist} We say that the degree distribution $P$ has {\em strongly finite mean} when there exist constants \ch{$\tau>2$} and $c>0$ such that
\eq \label{eq-ppowerlaw}
\sum_{i=k}^{\infty} P_i \leq c k^{-(\tau-1)}.
\en
\end{definition}
\ch{For technical reasons, we will assume, without loss of generality, that $\tau\in(2,3)$ in the rest of the paper.}
Note that all distributions $P$ where
\eq
\sum_{i=k}^{\infty} P_i = c k^{-(\tau-1)} L(k),
\en
for $c>0, \tau>2$ and some slowly varying function $L(k)$, have {\em strongly finite mean}, because by Potter's theorem (\cite[Lemma~2, p.277]{Fel71}) any slowly varying function $L(k)$ can be bounded above and below by an arbitrary small power of $k$. Also distributions which have a lighter tail than a power law, e.g. the Poisson distribution, have {\em strongly finite mean}.

\begin{definition}[Uniform sparsity]\label{ass-unisparse}
We say that the graph sequence $\{G_n\}_{n \geq 1}$ is {\em uniformly sparse} when, a.s.,
\eq
\lim_{\ell\rightarrow\infty} \limsup_{n\rightarrow\infty} \frac{1}{n} \sum_{i \in [n]} D_i \ind_{\{ D_i \geq \ell\}} = 0,
\en
where $D_i$ is the degree of vertex $i$ in $G_n$ and $\ind_{\mathcal{A}}$ denotes the indicator of the event $\mathcal{A}$.
\end{definition}
An immediate consequence of the local convergence and the uniform sparsity condition is, that, a.s.,
\begin{align}
\lim_{n\rightarrow\infty}\frac{|E_n|}{n} &=\lim_{n\rightarrow\infty} \frac{1}{2 n} \sum_{i\in[n]} \sum_{k=1}^\infty k \ind_{\{D_i=k\}}
=\frac12 \lim_{\ell \rightarrow \infty} \lim_{n\rightarrow\infty} \left(\sum_{k=1}^{\ell-1} k \frac{\sum_{i \in [n]} \ind_{\{D_i=k\}}}{n} + \frac{1}{n}\sum_{i\in[n]} D_i \ind_{\{D_i \geq \ell\}}\right) \nn\\
&= \frac12 \lim_{\ell \rightarrow \infty} \sum_{k=1}^{\ell-1} k P_k = \overline{P}/2<\infty.\label{eq-onedges}
\end{align}

\subsection{Main results}
We first investigate the thermodynamic limit of the pressure:
\begin{theorem}[Thermodynamic limit of the pressure]\label{thm-pressure}
Assume that the random graph sequence $\{G_n\}_{n\geq1}$ is locally tree-like with asymptotic degree distribution $P$, where $P$ has strongly finite mean, and is uniformly sparse. Then, for all $0\leq \beta <\infty$ and all $B\in\bbR$, the thermodynamic limit of the pressure exists, a.s., and equals
\eq
\lim_{n \rightarrow \infty} \psi_n(\beta,B) = \varphi(\beta, B),
\en
where, for $B<0$, $\varphi(\beta, B) = \varphi(\beta, -B)$, $\varphi(\beta, 0) = \lim_{B\downarrow 0}\varphi(\beta, B)$ and, for $B>0$,
\begin{align}
\varphi(\beta,B) =& \frac{\overline{P}}{2} \log \cosh(\beta) - \frac{\overline{P}}{2} \expec[ \log(1+\tanh(\beta)\tanh(h_1)\tanh(h_2))] \nn\\
&\qquad + \expec\left[\log\left(e^B \prod_{i=1}^{L} \{1+\tanh(\beta)\tanh(h_i)\} + e^{-B} \prod_{i=1}^{L} \{1-\tanh(\beta)\tanh(h_i)\} \right)\right], \label{eq-pressure}
\end{align}
where
\begin{itemize}
\item[(i)]
$L$ has distribution $P$;

\item[(ii)]
$\{h_i\}_{i\geq1}$ are i.i.d.\ copies of the fixed point $h^*=h^*(\beta,B)$ of the distributional recursion
\eq \label{eq-recursion}
h^{(t+1)} \stackrel{d}{=} B + \sum_{i=1}^{K_t} \atanh (\tanh(\beta)\tanh(h_i^{(t)})),
\en
where $h^{(0)} \equiv B$, $\{K_t\}_{t \geq 1}$, are i.i.d.\ with distribution $\rho$
and $\{h_i^{(t)}\}_{i\geq1}$ are i.i.d.\ copies of $h^{(t)}$ independent of $K_t$;

\item[(iii)]
$L$ and $\{h_i\}_{i\geq1}$ are independent.
\end{itemize}
\end{theorem}

The quantity $\varphi(\beta, B)$ can be seen as the infinite volume pressure of the Ising model on the random Bethe lattice, where every vertex has degree distributed as $P$ (cf.~\cite{Bax82} where the Ising model on the regular Bethe lattice is studied).

Various thermodynamic quantities can be computed by taking the proper derivative of the function $\varphi(\beta,B)$ as we shall show in the next theorem.
\begin{theorem}[Thermodynamic quantities]\label{thm-thermqphi}
Assume that the random graph sequence $\{G_n\}_{n\geq1}$ is locally tree-like with asymptotic degree distribution $P$, where $P$ has strongly finite mean, and is uniformly sparse. Then, for all $\beta\geq0$ and $B\neq0$, each of the following statements holds a.s.:
\begin{description}
\item[(a) Magnetization.]
Let $M_n(\beta,B) = \frac{1}{n} \sum_{i\in[n]} \left<\sigma_i\right>_{\mu_n}$ be the {\em magnetization} per vertex. Then, its thermodynamic limit exists and is given by
\eq
M(\beta, B) \equiv \lim_{n\rightarrow \infty} M_n(\beta,B) = \frac{\partial}{\partial B} \varphi(\beta,B).
\en

\item[(b) Internal energy.]
Let $U_n(\beta,B) = -\frac{1}{n} \sum_{(i,j)\in \ch{E_n}} \left<\sigma_i \sigma_j\right>_{\mu_n}$ be the {\em internal energy} per vertex. Then, its thermodynamic limit exists and is given by
\eq
U(\beta,B) \equiv \lim_{n\rightarrow \infty} U_n(\beta,B) = - \frac{\partial}{\partial \beta}\varphi(\beta, B).
\en

\item[(c) Susceptibility.] Let $\chi_n(\beta,B) = \frac{1}{n} \sum_{\ch{i,j\in [n]}}\left( \left<\sigma_i \sigma_j\right>_{\mu_n} -\left<\sigma_i\right>_{\mu_n}\left<\sigma_j\right>_{\mu_n}\right) = \frac{\partial M_n}{\partial B}(\beta,B)$ be the {\em susceptibility}. Then, its thermodynamic limit exists and is given by
\eq
\chi(\beta,B) \equiv \lim_{n\rightarrow \infty} \chi_n(\beta,B) = \frac{\partial^2}{\partial B^2}\varphi(\beta, B).
\en
\end{description}
\end{theorem}
The limits above hold for $\beta\geq0$ and $B\neq0$. From the physics
literature (see e.g. \cite{DorGolMen02,LeoVazVesZec02})
we expect that this model has a ferromagnetic phase transition at
$\beta_c = \atanh(1/\overline{\rho})$, i.e., the susceptibility
becomes infinite at $\beta=\beta_c$ in $B=0$. For $\beta<\beta_c$ the
functions above are continuous in $B$ for all $B$ and thus the limits
above also hold in this regime.

Another physical quantity studied in the physics literature is the {\em specific heat},
\eq
C_n(\beta,B) \equiv -\beta^2 \frac{\partial U_n}{\partial \beta}.
\en
Unfortunately, we were not able to prove that this converges to $\beta^2 \frac{\partial^2}{\partial \beta^2}\varphi(\beta, B)$, because we do not have convexity or concavity of the internal energy in $\beta$. We expect, however, that this limit also holds.

Taking the derivatives of Theorem~\ref{thm-thermqphi} we can also give explicit expressions for the magnetization and internal energy which have a physical interpretation:
\begin{corollary}[Explicit expressions for thermodynamic quantities]\label{cor-thermq}
Assume that the graph sequence $\{G_n\}_{n\geq1}$ is locally tree-like with asymptotic degree distribution $P$, where $P$ has strongly finite mean, and is uniformly sparse. Then, for all $\beta\geq0$ and $B\in\mathbb{R}$, each of the following statements holds a.s.:
\begin{description}
\item[(a) Magnetization.]
Let $\nu_{L+1}$ be the random Ising measure on a tree with $L+1$ vertices (one root and $L$ leaves), where $L$ has distribution $P$, defined by
\eq \label{eq-defnuL1}
\nu_{L+1}(\sigma) = \frac{1}{Z_{L+1}(\beta,h^*)} \exp\left\{ \beta \sum_{i=1}^L \sigma_0 \sigma_i + B \sigma_0 + \sum_{i=1}^{L} h_i \sigma_i \right\},
\en
where $\{h_i\}_{i\geq1}$ are i.i.d.\ copies of $h^*$, independent of $L$. Then, the thermodynamic limit of the {\em magnetization} per vertex is given by
\eq
M(\beta, B) = \expec\left[\big<\sigma_0\big>_{\nu_{L+1}}\right],
\en
where the expectation is taken over $L$ and $\{h_i\}_{i\geq1}$. More explicitly,
\eq
M(\beta,B) = \expec\left[ \tanh\left(B+\sum_{i=1}^{L} \atanh(\tanh(\beta) \tanh(h_i))\right)\right].
\en

\item[(b) Internal energy.]
Let $\nu'_2$ be the random Ising measure on one edge, defined by
\eq \label{eq-defnu2}
\nu'_2(\sigma) = \frac{1}{Z_2(\beta,h_1,h_2)} \exp\left\{ \beta \sigma_1 \sigma_2 + h_1 \sigma_1 + h_2 \sigma_2\right\},
\en
where $h_1$ and $h_2$ are i.i.d.\ copies of $h^*$. Then the thermodynamic limit of the {\em internal energy} per vertex is given by
\eq
U(\beta,B) = -\frac{\overline{P}}{2} \expec\left[\big<\sigma_1\sigma_2\big>_{\nu'_2}\right],
\en
where the expectation is taken over $h_1$ and $h_2$. More explicitly,
\eq
U(\beta,B) = -\frac{\overline{P}}{2} \expec\left[\frac{\tanh(\beta)+\tanh(h_1)\tanh(h_2)}{1+\tanh(\beta) \tanh(h_1)\tanh(h_2)}\right].
\en
\end{description}
\end{corollary}

Note that the magnetization and internal energy are {\em local} observables, i.e., they are {\em spin} or {\em edge} variables averaged out over the graph. This is not true for the susceptibility, which is an average over pairs of spins, and hence we were not able to give an explicit expression for this quantity.

\subsection{Discussion}
We study the Ising model on a random graph, which gives rise to a model with double randomness. Still, in the thermodynamic limit, the pressure is essentially deterministic. This is possible, because it suffices to study the Ising model on the local neighborhood of a uniformly chosen vertex. This local neighborhood converges by our assumptions to the tree $\calT(P,\rho,\infty)$, and it thus suffices to study the Ising model on this limiting object. An analysis of this kind is therefore known as the {\em objective method} introduced by Aldous and Steele in \cite{AldSte04}.

The assumption that the graph converges locally to a homogeneous tree holds in a wide range of random graph models, among which the configuration model and the Erd{\H o}s-R\'enyi random graph, as we now explain.

In the {\em configuration model}, a random graph $G_n$ is constructed as follows. Let $\{D_i\}_{i=1}^n$ be a sequence of i.i.d.\ random variables with a certain {\em degree distribution} $P$. Let vertex $i\in [n]$ be a vertex with $D_i$ half-edges, also called {\em stubs}, attached to it, i.e., vertex $i$ has degree $D_i$. Let $L_n=\sum_{i=1}^n D_i$ be the total degree, which we assume to be even in order to be able to construct a graph. When $L_n$ is odd we will increase the degree of $D_n$ by 1. For $n$ large, this will hardly change the results and we will therefore ignore this effect. Now connect one of the half-edges uniformly at random to one of the remaining $L_n - 1$ half-edges. Repeat this procedure until all half-edges have been connected.

Dembo and Montanari studied a slightly different version of the configuration model in~\cite{DemMon10a}, where the degrees of the vertices are {\em deterministic} instead of random. They prove that in that case, when the empirical degree distribution has finite mean, the graph sequence is also locally tree-like and uniformly sparse. Their proof can easily be adapted to show that this also holds for the above version of the configuration model using the strong law of large numbers.

In~\cite{DemMon10a}, it was also shown that the {\em Erd{\H o}s-R\'enyi} random graph is uniformly sparse and is locally tree-like with an asymptotic degree distribution that is Poisson distributed (and note that for a Poisson distribution $P=\rho$). Therefore, clearly, the Erd{\H o}s-R\'enyi random graph has a finite mean degree distribution.

That these results hold for a wide variety of random graph models is not a surprise. It is believed that the behavior of networks shows a great universality. Distances in random graph models, for example, also show a remarkably universal behavior. See, e.g.,~\cite{HofHoo08} for an overview of results on distances in power-law random graphs. These distances mainly depend on the power-law exponent and not on other details of the graph.
Note, however, that the results above only apply to graphs that converge locally to a {\em homogeneous} tree and thus, for instance, not for many inhomogeneous random graphs studied in \cite{BolJanRio07} where the local structure is a multi-type Galton-Watson branching process instead. Certain parts of our proof easily extend to this case.

In this paper we study {\em smooth} observables, we defer the investigation of the critical nature to a later paper. There we will study the behavior around the critical value $\beta_c$ where certain quantities (e.g. the susceptibility) have singularities. Of special interest is the critical behavior when $\tau \in (2,3)$, where $\beta_c=0$, i.e., where the system is always in the ferromagnetic regime for any finite temperature.

\subsection{Overview of the proof and organization of the paper}
In this section, we give an overview of the proof of Theorem~\ref{thm-pressure}, and reduce it to the proofs of Propositions~\ref{prop-recursion},~\ref{prop-phiindeph} and~\ref{prop-dpressure} below. Proposition~\ref{prop-recursion} establishes that the recursive relation that gives the field acting on the root of the infinite tree $\calT(P,\rho,\infty)$ is well-defined, in the sense that the recursion admits a unique fixed point $h^*$. Proposition~\ref{prop-phiindeph} is instrumental to control the implicit dependence of the pressure of the random Bethe lattice $\varphi(\beta,B)$ on the inverse temperature $\beta$ via the field $h^*$. This is used in Proposition~\ref{prop-dpressure} which proves that the derivative of the pressure with respect to $\beta$, namely minus the {\em internal energy}, converges in the thermodynamic limit to the derivative of $\varphi(\beta,B)$. We also clearly indicate how our proof deviates from that by Dembo and Montanari in \cite{DemMon10}.

We will first analyze the case where $B>0$ and deal with $B\leq0$ later. We start by investigating the distributional recursion~\eqref{eq-recursion}:
\begin{proposition}[Tree recursion]\label{prop-recursion}
Let $B>0$ and let $\{K_t\}_{t \geq 1}$ be i.i.d.\ according to some distribution $\rho$ and assume that $K_1<\infty$, a.s. Consider the sequence of random variables $\{h^{(t)}\}_{t \geq 0}$ defined by $h^{(0)} \equiv B$ and, for $t\geq0$, by~\eqref{eq-recursion}.
Then, the distributions of $h^{(t)}$ are stochastically monotone and $h^{(t)}$ converges in distribution to the {\em unique} fixed point $h^*$ of the recursion~\eqref{eq-recursion} that is supported on $[0,\infty)$.
\end{proposition}
We can now investigate the thermodynamic limit of the pressure. By the fundamental theorem of calculus,
\begin{align}
\lim_{n\rightarrow\infty} \psi_n(\beta,B)
&= \lim_{n\rightarrow\infty}\left[\psi_n(0,B) +  \int_{0}^{\beta} \frac{\partial}{\partial \beta'}\psi_n(\beta',B) \dint \beta' \right]\nn\\
&= \lim_{n\rightarrow\infty}\left[\psi_n(0,B) +  \int_{0}^{\varepsilon} \frac{\partial}{\partial \beta'}\psi_n(\beta',B) \dint \beta' + \int_{\varepsilon}^{\beta} \frac{\partial}{\partial \beta'}\psi_n(\beta',B) \dint \beta' \right],
\end{align}
for any $0<\varepsilon<\beta$. For all $n \geq 1$, we have that
\eq
\psi_n(0,B) = \log(2 \cosh(B)) = \varphi(0,B),
\en
so this is also true for $n\rightarrow\infty$.

By the uniform sparsity of $\{G_n\}_{n\geq1}$,
\eq
\left|\frac{\partial}{\partial \beta}\psi_n(\beta,B)\right| = \left|\frac{1}{n}\sum_{(i,j)\in E_n} \big<\sigma_i \sigma_j\big>_{\mu_n}\right| \leq {\frac{|E_n|}{n}} \leq c,
\en
for some constant $c$. Thus, uniformly in $n$,
\eq
\left|\int_{0}^{\varepsilon} \frac{\partial}{\partial \beta'}\psi_n(\beta',B) \dint \beta'\right| \leq c \varepsilon.
\en
Using the boundedness of the derivative for $\beta'\in[\varepsilon,\beta]$, we also have that
\eq
\lim_{n\rightarrow\infty} \int_{\varepsilon}^{\beta} \frac{\partial}{\partial \beta'}\psi_n(\beta',B) \dint \beta'
= \int_{\varepsilon}^{\beta} \lim_{n\rightarrow\infty} \frac{\partial}{\partial \beta'}\psi_n(\beta',B) \dint \beta'.
\en
For $\beta>0$, we will show that the partial derivative with respect to $\beta$ of $\psi_n(\beta,B)$ converges to the partial derivative with respect to $\beta$ of $\varphi(\beta,B)$. For this, we need that we can in fact ignore the dependence of $h^*$ on $\beta$ when computing the latter derivative as we shall show first:
\begin{proposition}[Dependence of $\varphi$ on $(\beta,B)$ via $h^*$]\label{prop-phiindeph}
Assume that the distribution $P$ has strongly finite mean. Fix $B_1,B_2>0$ and $0 < \beta_1, \beta_2 <\infty$. Let $h_{1}^*$ and $h_{2}^*$ be the fixed points of \eqref{eq-recursion} for $(\beta_1,B_1)$ and $(\beta_2,B_2)$, respectively. Let $\varphi_{h^*}(\beta,B)$ be defined as in~\eqref{eq-pressure} with $\{h_i\}_{i\geq1}$ replaced by i.i.d.\ copies of the specified $h^*$. Then,
\begin{description}
\item[(a)] For $B_1=B_2$, there exists a $\lambda_1<\infty$ such that
\eq
|\varphi_{h_{1}^*}(\beta_1, B_1) - \varphi_{h_{2}^*}(\beta_1, B_1)| \leq \lambda_1 |\beta_1-\beta_2|^{\tau-1}.
\en
\item[(b)] For $\beta_1=\beta_2$, there exists a $\lambda_2<\infty$ such that
\eq
|\varphi_{h_{1}^*}(\beta_1, B_1) - \varphi_{h_{2}^*}(\beta_1, B_1)| \leq \lambda_2 |B_1-B_2|^{\tau-1}.
\en
\end{description}
\end{proposition}
\ch{Note that this proposition only holds if $\tau\in(2,3)$. For $\tau>3$, the exponent $\tau-1$ can be improved to $2$, as is shown in~\cite{DemMon10}, but this is not of importance to the proof.} We need part (b) of the proposition above later in the proof of Corollary~\ref{cor-thermq}.

\begin{proposition}[Convergence of the internal energy]\label{prop-dpressure}
Assume that the graph sequence $\{G_n\}_{n\geq1}$ is locally tree-like with asymptotic degree distribution $P$, where $P$ has strongly finite mean, and is uniformly sparse. Let $\beta>0$. Then, a.s.,
\eq
\lim_{n\rightarrow\infty} \frac{\partial}{\partial \beta} \psi_n(\beta,B) = \frac{\partial}{\partial \beta} \varphi(\beta,B),
\en
where $\varphi(\beta,B)$ is given in \eqref{eq-pressure}.
\end{proposition}
By Proposition~\ref{prop-dpressure} and bounded convergence,
\eq
\int_{\varepsilon}^{\beta} \lim_{n\rightarrow\infty} \frac{\partial}{\partial \beta'}\psi_n(\beta',B) \dint \beta' = \int_{\varepsilon}^{\beta} \frac{\partial}{\partial \beta'}\varphi(\beta',B) \dint \beta'=\varphi(\beta,B)-\varphi(\varepsilon,B),
\en
again by the fundamental theorem of calculus.

Observing that $0\leq \tanh(h^*) \leq 1$, one can show that, by dominated convergence, $\varphi(\beta,B)$ is right-continuous in $\beta=0$. Thus, letting $\varepsilon \downarrow 0$,
\begin{align}
\lim_{n\rightarrow\infty} \psi_n(\beta,B) &= \lim_{\varepsilon\downarrow 0} \lim_{n\rightarrow\infty}\left[\psi_n(0,B) +  \int_{0}^{\varepsilon} \frac{\partial}{\partial \beta'}\psi_n(\beta',B) \dint \beta' + \int_{\varepsilon}^{\beta} \frac{\partial}{\partial \beta'}\psi_n(\beta',B) \dint \beta' \right]\nn\\
&= \varphi(0,B) + \lim_{\varepsilon\downarrow 0}\left(\varphi(\beta,B)-\varphi(\varepsilon,B)\right) = \varphi(\beta,B),
\end{align}
which completes the proof for $B>0$.

The Ising model with $B<0$ is equivalent to the case $B>0$, because one can multiply all spin variables $\{\sigma_i\}_{i\in[n]}$ and $B$ with $-1$ without changing Boltzmann distribution~\eqref{eq-boltzmann}. Furthermore, note that,
\eq
\left|\frac{\partial}{\partial B} \psi_n(\beta,B)\right|=\left|\frac{1}{n} \sum_{i\in[n]} \left<\sigma_i\right>_{\mu_n}\right| \leq 1,
\en
so that $B\mapsto\psi_n(\beta,B)$ is uniformly Lipschitz continuous with Lipschitz constant one. Therefore,
\eq
\lim_{n\rightarrow\infty} \psi_n(\beta,0) = \lim_{n\rightarrow\infty} \lim_{B\downarrow0} \psi_n(\beta,B)
=\lim_{B\downarrow0}\lim_{n\rightarrow\infty}  \psi_n(\beta,B) = \lim_{B\downarrow0} \varphi(\beta,B).
\en
\qed

The proof given above follows the line of argument in~\cite{DemMon10}, but in order to prove Propositions~\ref{prop-recursion},~\ref{prop-phiindeph} and~\ref{prop-dpressure} we have to make substantial changes to generalize the proof to the infinite variance case.

To prove Proposition~\ref{prop-recursion}, we adapt the proof of Dembo and Montanari by taking the actual forward degrees into account, instead of using Jensen's inequality to replace them by expected forward degrees, which are potentially infinite. This also makes a separate analysis of nodes that have zero offspring superfluous, which considerably simplifies the analysis.

The proof of Proposition~\ref{prop-phiindeph}(a) is somewhat more elaborate, because we have to distinguish between the cases where $L$ in~\eqref{eq-pressure} is small or large, but the techniques used remain similar. By, again, taking into account the actual degrees more precisely, the analysis is simplified however: we, for example, do not rely on the exponential decay of the correlations. Part (b) of this proposition is new and can be proved with similar techniques. The proof of Proposition~\ref{prop-dpressure} is proven in a similar way as in~\cite{DemMon10}.

The remainder of this paper is organized as follows. First we shall review some preliminaries on Ising models in Section~\ref{sec-preli}. Next, we shall study the tree recursion of~\eqref{eq-recursion} and prove Proposition~\ref{prop-recursion} in Section~\ref{sec-recursion} and Proposition~\ref{prop-phiindeph} in Section~\ref{sec-phiindeph}. Finally, in Section~\ref{sec-dpressure}, we shall prove Proposition~\ref{prop-dpressure}. In Section~\ref{sec-thermq} we shall study the thermodynamic quantities to prove Corollary~\ref{cor-thermq}.

\section{Preliminaries}\label{sec-preli}

The first result on ferromagnetic Ising models we will heavily rely on is the \emph{Griffiths, Kelly, Sherman (GKS) inequality}, which gives various monotonicity properties:
\begin{lemma}[GKS inequality]\label{lem-Griffith}
Consider two Ising measures $\mu$ and $\mu'$ on graphs $G=(V,E)$ and $G'=(V,E')$, with inverse temperatures $\beta$ and $\beta'$ and external fields $\underline{B}$ and $\underline{B}'$, respectively.
If $E \subseteq E'$, $\beta \leq \beta'$ and $0 \leq B_i \leq B_i'$ for all $i \in V$, then, for any $U \subseteq V$,
\eq
0 \leq \big<\prod_{i \in U} \sigma_i\big>_{\mu} \leq \big<\prod_{i \in U} \sigma_i\big>_{\mu'}.
\en
\end{lemma}
A weaker version of this inequality was first proved by Griffiths \cite{Gri67a} and later generalized by Kelly and Sherman \cite{KelShe68}. The second result on ferromagnetic Ising models is an inequality by Griffiths, Hurst and Sherman \cite{GriHurShe70} which shows the concavity of the magnetization in the external (positive) magnetic fields.
\begin{lemma}[GHS inequality]\label{lem-GHS} Let $\beta \geq 0$ and $B_i \geq 0$ for all $i \in V$. Denote by
\eq
m_j(\underline{B})=\mu(\{\sigma: \sigma_j=+1\}) - \mu(\{\sigma: \sigma_j=-1\})
\en
the magnetization of vertex $j$ when the external fields at the vertices are $\underline{B}$. Then, for any \ch{three} vertices $j,k,l \in V$,
\eq
\frac{\partial^2}{\partial B_k \partial B_\ell}m_j (\underline{B}) \leq 0.
\en
\end{lemma}
The final preliminary observation we need is a lemma that reduces the computation of the Ising measure on a tree to the computation of Ising measures on subtrees:
\begin{lemma}[Pruning trees]\label{lem-treepruning}
For $U$ a subtree of a finite tree $\calT$, let $\partial U$ be the subset of vertices of $U$ that connect to a vertex in $W\equiv \calT \setminus U$. Denote by $\left<\sigma_u\right>_{\mu_{W,u}}$ the magnetization of vertex $u \in \partial U$ of the Ising model on $W \cup \{u\}$.
Then, the marginal Ising measure on $U$, $\mu_U^{\mathcal{T}}$, is the same as the Ising measure on $U$ with magnetic fields
\begin{equation}
B_u'=\left\{\begin{array}{ll}\atanh(\left<\sigma_u\right>_{\mu_{W,u}}), \qquad & \qquad u \in \partial U,\\
                            B_u, \qquad & \qquad u \in U \setminus \partial U.
            \end{array}\right.
\end{equation}
\end{lemma}
The proof of this lemma follows from a direct application of the Boltzmann distribution given in~\eqref{eq-boltzmann}, see \cite[Lemma~4.1]{DemMon10}.

\section{Tree recursion: proof of Proposition~\ref{prop-recursion}}\label{sec-recursion}

To prove Proposition~\ref{prop-recursion}, we will first study the Ising model on a tree with $\ell$ generations, $\calT(\ell)$, with either $+$~or free boundary conditions, where the Ising models on the tree $\calT(\ell)$ with $+$/free boundary conditions are defined by the Boltzmann distributions
\eq
\mu^{\ell,+}(\sigma) = \frac{1}{Z^{\ell,+}(\beta, \underline{B})} \exp \left\{\beta \sum_{(i,j) \in \calT(\ell)} \sigma_i \sigma_j + \sum_{i \in \calT(\ell)} B_i \sigma_i\right\} \ch{\ind_{\{\sigma_i=+1, \textrm{\ for all\ } i\in\partial \calT(\ell)\}}},
\en
and
\eq
\mu^{\ell,f}(\sigma) = \frac{1}{Z^{\ell,f}(\beta, \underline{B})} \exp \left\{\beta \sum_{(i,j) \in \calT(\ell)} \sigma_i \sigma_j + \sum_{i \in \calT(\ell)} B_i \sigma_i\right\},
\en
respectively, where $Z^{\ell,+/f}$ are the proper normalization factors and $\partial \calT(\ell)$ denotes the vertices in the $\ell$-th generation of $\calT(\ell)$. In the next lemma we will show that the effect of these boundary conditions vanishes when $\ell\rightarrow\infty$. \ch{This lemma is a generalization of \cite[Lemma~4.3]{DemMon10}, where this result was proved in expectation for graphs with a finite-variance degree distribution. This generalization is possible by taking the degrees into account more precisely, instead of using Jensen's inequality to replace them by average degrees. This also simplifies the proof.}

We will then show that the recursion~\eqref{eq-recursion} has a fixed point and use a coupling with the root magnetization in trees and Lemma~\ref{lem-45} to show that this fixed point does not depend on the initial distribution $h^{(0)}$, thus showing that \eqref{eq-recursion} has a {\em unique} fixed point.
\begin{lemma}[Vanishing effect of boundary conditions]\label{lem-45}
Let $m^{\ell,+/f}(\underline{B})$ denote the root magnetization given $\calT(\ell)$ with external field per vertex $B_i \geq B_{\min} > 0$ when the tree has $+$/free boundary conditions. Assume that the forward degrees satisfy $\Delta_i < \infty$ a.s., for all $i\in \calT(\ell-1)$. Let $0\leq \beta \leq \beta_{\max} < \infty$. Then, there exists an $M=M(\beta_{\max},B_{\min})<\infty$ such that, a.s.,
\eq
m^{\ell,+}(\underline{B}) - m^{\ell,f}(\underline{B}) \leq \frac{M}{\ell},
\en
for all $\ell \geq 1$.
\end{lemma}
\begin{rem}
Lemma~\ref{lem-45} is extremely general. For example, it also applies to trees arising from multitype branching processes.
\end{rem}
\begin{proof}
The lemma clearly holds for $\beta=0$, so we assume that $\beta>0$ in the remainder of the proof.

Denote by $m^{\ell}(\underline{B},\underline{H})$ the root magnetization given $\mathcal{T}(\ell)$ with free boundary conditions, when the external field on the vertices $i\in \partial \mathcal{T}(\ell)$ is $B_i+H_i$ and $B_i$ on all other vertices $i\in \mathcal{T}(\ell-1)$. Condition on the tree $\calT(\ell)$ and assume that the tree $\calT(\ell)$ is finite, which is true a.s., so that we can use Lemma~\ref{lem-treepruning}. Thus, for $1\leq k\leq \ell$,
\eq \label{eq-GI1}
m^{k,+}(\underline{B}) \equiv m^{k}(\underline{B}, \infty) = m^{k-1}(\underline{B}, \{\beta \Delta_i\}),
\en
where $\Delta_i$ is the forward degree of vertex $i \in \partial \calT(k-1)$. By the GKS inequality
\eq \label{eq-GI2}
m^{k-1}(\underline{B}, \{\beta \Delta_i\}) \leq m^{k-1}(\underline{B}, \infty).
\en

Since the magnetic field at all vertices in $\partial \calT(k)$ is at least $B_{\min}$ we can write, using Lemma~\ref{lem-treepruning} and  the GKS inequality, that
\eq \label{eq-GI3}
m^{k,f}(\underline{B}) \equiv m^{k}(\underline{B}, 0) \geq m^{k-1}(\underline{B},  \xi\{ \Delta_i\}),
\en
where
\eq \label{eq-defxi}
\xi = \xi(\beta, B_{\min}) = \atanh (\tanh(\beta)\tanh(B_{\min})).
\en
This inequality holds with equality when $B_i=B_{\min}$ for all $i \in \partial \calT(k)$. Using the GKS inequality again, we have that
\eq \label{eq-GI4}
m^{k-1}(\underline{B},  \xi\{ \Delta_i\}) \geq m^{k-1}(\underline{B}, 0).
\en

Note that $0\leq \xi(\beta,B_{\min}) \leq \beta$. Since $H \mapsto m^{k}(\underline{B}, H \{\Delta_i\})$ is concave in $H$ because of the GHS inequality, we have that
\eq \label{eq-GHS1}
m^{k-1}(\underline{B}, \beta \{\Delta_i\}) - m^{k-1}(\underline{B},0) \leq M \left(m^{k-1}(\underline{B}, \xi \{\Delta_i\}) - m^{k-1}(\underline{B},0)\right),
\en
where
\eq
M = M(\beta_{\max},B_{\min}) = \sup_{0 < \beta \leq \beta_{\max}}\frac{\beta}{\xi(\beta,B_{\min})}<\infty.
\en
Thus, we can rewrite $m^{k,+}(\underline{B})$ using~\eqref{eq-GI1} and bound $m^{k,f}(\underline{B})$ using~\eqref{eq-GI3} and~\eqref{eq-GI4}, to obtain
\eq
m^{k,+}(\underline{B}) - m^{k,f}(\underline{B}) \leq m^{k-1}(\underline{B},\beta \{\Delta_i\}) - m^{k-1}(\underline{B},0).
\en
By~\eqref{eq-GHS1}, we then have that
\eq \label{eq-concavity}
m^{k,+}(\underline{B}) - m^{k,f}(\underline{B}) \leq M \left(m^{k-1}(\underline{B},\xi \{\Delta_i\}) - m^{k-1}(\underline{B},0)\right) \leq M \left(m^{k}(\underline{B},0) - m^{k-1}(\underline{B},0)\right),
\en
where we have used~\eqref{eq-GI3} in the last inequality.

By~\eqref{eq-GI1} and~\eqref{eq-GI2}, $m^{k,+}(\underline{B})$ is non-increasing in $k$ and, by~\eqref{eq-GI3} and~\eqref{eq-GI4}, $m^{k,f}(\underline{B})$ is non-decreasing in $k$. Thus, by summing the inequality in~\eqref{eq-concavity} over $k$, we get that
\begin{align}
\ell \left(m^{\ell,+}(\underline{B})-m^{\ell,f}(\underline{B})\right)
& \leq \sum_{k=1}^{\ell} \left(m^{k,+}(\underline{B})-m^{k,f}(\underline{B})\right)
\leq M \sum_{k=1}^{\ell} \left(m^{k}(\underline{B},0)-m^{k-1}(\underline{B},0)\right) \nn\\
&= M \left(m^{\ell}(\underline{B},0)-m^{0}(\underline{B},0)\right)
\leq M,
\end{align}
since $0\leq m^{\ell/0}(\underline{B},0) \leq 1$.
\end{proof}

We are now ready to prove Proposition~\ref{prop-recursion}.
\begin{proof}[Proof of Proposition~\ref{prop-recursion}]
Condition on the tree $\calT(\rho,\infty)$. Then $h^{(t)} \equiv \atanh(m^{t,f}(B))$ satisfies the recursive distribution~\eqref{eq-recursion} because of Lemma~\ref{lem-treepruning}. Since, by the GKS inequality, $m^{t,f}(B)$, and hence also $h^{(t)}$, are monotonically increasing in $t$, we have that $B=h^{(0)} \leq h^{(t)} \leq B+D_0<\infty$ for all $t\geq0$, where $D_0$ is the degree of the root. So, $h^{(t)}$ converges to some limit $\underline{h}$. Since this holds a.s. for any tree $\calT(\rho,\infty)$, the distribution of $\underline{h}$ also exists and one can show that this limit is a fixed point of~\eqref{eq-recursion} (see \cite[Proof of \ch{Lemma~2.3}]{DemMon10}).

In a similar way, $h^{(t,+)} \equiv \atanh(m^{t,+}(B))$ satisfies~\eqref{eq-recursion} when starting with $h^{(0,+)}=\infty$. Then, $h^{(t,+)}$ is monotonically decreasing and, for $t\geq1$, $B\leq h^{(t)} \leq B+D_0<\infty$, so $h^{(t,+)}$ also converges to some limit $\overline{h}$.

Let $h$ be a fixed point of~\eqref{eq-recursion}, condition on this $h$ and let $h^{(0,*)}=h$. Then $h^{(t,*)}$ converges as above to a limit $h^*$ say, when applying~\eqref{eq-recursion}. Note that $h^{(0)} \leq h^{(0,*)} \leq h^{(0,+)}$. Coupling so as to have the same $\{K_t\}_{t\geq1}$ while applying the recursion~\eqref{eq-recursion}, this order is preserved by the GKS inequality, so that $h^{(t)} \leq h^{(t,*)} \leq h^{(t,+)}$ for all $t\geq0$. By Lemma~\ref{lem-45},
\eq
|\tanh(h^{(t)}) - \tanh(h^{(t,+)})| = |m^{t,f}(B)-m^{t,+}(B)| \rightarrow 0, \qquad \textrm{for $t\rightarrow\infty$}.
\en
Since the above holds a.s. for any tree $\calT(\rho,\infty)$ and any realization of $h^*$, the distributions of $\underline{h}, \overline{h}$ and $h^*$ are equal, and, since $h$ is a fixed point of~\eqref{eq-recursion}, are all equal in distribution to $h$.
\end{proof}

\section{Dependence of $\varphi$ on $(\beta,B)$ via $h^*$: proof of Proposition~\ref{prop-phiindeph}}\label{sec-phiindeph}

We will now prove Proposition~\ref{prop-phiindeph} by first bounding the dependence of $\varphi$ on $h^*$ in Lemma~\ref{lem-depphih} and subsequently bounding the dependence of $h^*$ on $\beta$ and $B$ in Lemmas~\ref{lem-dephbeta} and~\ref{lem-dephB} respectively.
\begin{lemma}[Dependence of $\varphi$ on $h^*$]\label{lem-depphih}
Assume that distribution $P$ has strongly finite mean. Fix $B_1,B_2>0$ and $0 < \beta_1, \beta_2 <\infty$. Let $h_{1}^*$ and $h_{2}^*$ be the fixed points of \eqref{eq-recursion} for $(\beta_1,B_1)$ and $(\beta_2,B_2)$, respectively. Let $\varphi_{h^*}(\beta,B)$ be defined as in~\eqref{eq-pressure} with $\{h_i\}_{i\geq1}$ replaced by i.i.d.\ copies of the specified $h^*$. Then, for some $\lambda < \infty$,
\eq
|\varphi_{h_{1}^*}(\beta_1, B_1) - \varphi_{h_{2}^*}(\beta_1, B_1)| \leq \lambda \|\tanh(h_{1}^*)-\tanh(h_{2}^*)\|_{\rm \sss MK}^{\tau-1},
\en
where $\|X-Y\|_{\rm \sss MK}$ denotes the Monge-Kantorovich-Wasserstein distance between random variables $X$ and $Y$, i.e., $\|X-Y\|_{\rm \sss MK}$ is the infimum of $\expec[ |\hat{X}-\hat{Y}|]$ over all couplings $(\hat{X},\hat{Y})$ of $X$ and $Y$.
\end{lemma}

\begin{proof}
Let $X_i$ and $Y_i$ be i.i.d.\ copies of $X=\tanh(h_{1}^*)$ and $Y=\tanh(h_{2}^*)$ respectively and also independent of $L$. When $\|X-Y\|_{\rm \sss MK}=0$ or $\|X-Y\|_{\rm \sss MK}=\infty$, the statement in the lemma clearly holds. Thus, without loss of generality,
we fix $\gamma>1$ and assume that $(X_i,Y_i)$ are i.i.d.\ pairs, independent of $L$, that are coupled in such a way that $\expec|X_i - Y_i| \leq \gamma \|X-Y\|_{\rm \sss MK}<\infty$.

Let $\widehat{\beta}=\tanh(\beta_1)$ and, for $\ell\geq 2$,
\eq \label{eq-defFl}
F_\ell(x_1,\ldots,x_\ell) = \log\left\{ e^B \prod_{i=1}^\ell (1+\widehat{\beta} x_i) + e^{-B}\prod_{i=1}^\ell (1-\widehat{\beta} x_i)\right\} -\frac{1}{\ell-1}\sum_{1\leq i < j \leq \ell} \log(1+\widehat{\beta} x_i x_j),
\en
and let
\begin{align}\label{eq-defF1}
F_1(x_1,x_2) = \frac12 \Big(\log &\big(e^B(1+\widehat{\beta}x_1)+e^{-B}(1-\widehat{\beta}x_1)\big) \nn\\
&+ \log \big(e^B(1+\widehat{\beta} x_2)+e^{-B}(1-\widehat{\beta} x_2)\big)-\log(1+\widehat{\beta} x_1 x_2)\Big).
\end{align}
Then, with $L$ having distribution $P$,
\eq
\varphi_{h_{1}^*}(\beta_1, B_1) = F_0 + \expec[F_L(X_1,\ldots,X_{\max\{2,L\}})] \quad \text{and} \quad \varphi_{h_{2}^*}(\beta_1, B_1) = F_0+\expec[F_L(Y_1,\ldots,Y_{\max\{2,L\}})],
\en
for some constant $F_0$ that is independent of $\beta$ and $B$. In the remainder of the proof we will assume that $F_1$ is defined as in~\eqref{eq-defFl}. The proof, however, also works for $F_1$ as defined in~\eqref{eq-defF1}.

We will split the absolute difference between $\varphi_{h_{1}^*}(\beta_1, B_1)$ and $\varphi_{h_{2}^*}(\beta_1, B_1)$ into two parts depending on whether $L$ is small or large, i.e., for some constant $\theta > 0$ to be chosen later on, we split
\begin{align}
\Big| \expec\big[ F_L (Y_1,\ldots,Y_L) - F_L(X_1,\ldots,X_L) \big] \Big| &\leq \Big| \expec\big[ (F_L (Y_1,\ldots,Y_L) - F_L(X_1,\ldots,X_L)) \ind_{\{L\leq \theta\}} \big] \Big| \nn\\
&\qquad +\Big| \expec\big[ (F_L (Y_1,\ldots,Y_L) - F_L(X_1,\ldots,X_L)) \ind_{\{L > \theta\}}\big] \Big|.
\end{align}
Note that
\begin{align}
F_\ell (Y_1,\ldots,Y_\ell) - F_\ell(X_1,\ldots,X_\ell) &= \int_{0}^{1} \frac{d}{ds} F_\ell(sY_1 + (1-s)X_1, \ldots, sY_\ell + (1-s)X_\ell)\Big|_{s=t} \dint t \nn\\
&= \int_{0}^{1} \sum_{i=1}^{\ell} (Y_i-X_i) \frac{\partial F_\ell}{\partial x_i} (tY_1 + (1-t)X_1, \ldots, tY_\ell + (1-t)X_\ell) \dint t \nn\\
&= \sum_{i=1}^{\ell} (Y_i - X_i) \int_{0}^{1} \frac{\partial F_\ell}{\partial x_i} (tY_1 + (1-t)X_1, \ldots, tY_\ell + (1-t)X_\ell) \dint t.
\end{align}
As observed in \cite[Corollary~6.3]{DemMon10}, $\frac{\partial F_\ell}{\partial x_i}$ is uniformly bounded, so that
\eq
\Big|F_\ell (Y_1,\ldots,Y_\ell) - F_\ell(X_1,\ldots,X_\ell)\Big| \leq \lambda_1 \sum_{i=1}^{\ell} |Y_i - X_i|,
\en
where $\lambda_1$ is allowed to change from line to line. Hence,
\begin{align}
\Big| \expec\big[ (F_L (Y_1,\ldots,Y_L) - F_L(X_1,\ldots,X_L)) \ind_{\{L > \theta\}}\big] \Big| &\leq \expec\left[\sum_{i=1}^{L} |Y_i - X_i| c_1 \ind_{\{L > \theta\}}]\right] \nn\\
&\leq \lambda_1 \|X-Y\|_{\rm \sss MK}\expec[L \ind_{\{L>\theta\}}] .
\end{align}
We compute, using that $L\geq0$,
\begin{align}
\expec[L \ind_{\{L>\theta\}}] &= \sum_{x=1}^{\infty}\prob[L \ind_{\{L>\theta\}}\geq x] = \sum_{x=1}^{\theta+1}\prob[L \ind_{\{L>\theta\}}\geq x] +\sum_{x=\theta+2}^{\infty}\prob[L \ind_{\{L>\theta\}}\geq x] \nn\\
&= \sum_{x=1}^{\theta+1}\prob[L \geq \theta+1] +\sum_{x=\theta+2}^{\infty}\prob[L \geq x]\leq (\theta+1)\cdot c(\theta+1)^{-(\tau-1)} +\sum_{x=\theta+2}^{\infty}c x^{-(\tau-1)} \nn\\
&\leq \lambda_1 \theta^{-(\tau-2)},
\end{align}
so that
\eq \label{eq-boundLbig}
\Big| \expec\big[ (F_L (Y_1,\ldots,Y_L) - F_L(X_1,\ldots,X_L)) \ind_{\{L > \theta\}}\big] \Big| \leq \lambda_1 \|X-Y\|_{\rm \sss MK} \theta^{-(\tau-2)}.
\en

By the fundamental theorem of calculus, we can also write
\eq
F_\ell (Y_1,\ldots,Y_\ell) - F_\ell(X_1,\ldots,X_\ell) = \sum_{i=1}^{\ell} \Delta_i F_\ell + \sum_{i\neq j}^{\ell} (Y_i-X_i)(Y_j-X_j)f_{ij}^{(\ell)},
\en
with
\eq\label{eq-deltaifl}
\Delta_i F_\ell = (Y_i-X_i) \int_0^1 \frac{\partial F_\ell}{\partial x_i} (X_1,\ldots,tY_i+(1-t)X_i,\ldots,X_\ell)\dint t,
\en
and
\eq
f_{ij}^{(\ell)} = \int_0^1 \int_0^t \frac{\partial^2 F_\ell}{\partial x_i \partial x_j} (sY_1 + (1-s)X_1,\ldots, sY_i + (1-s)X_i,\ldots,sY_\ell + (1-s)X_\ell)\dint s \dint t.
\en
Therefore,
\begin{align}
\Big| \expec\big[ (F_L (Y_1,\ldots,Y_L) - F_L(X_1,\ldots,X_L)) \ind_{\{L\leq \theta\}} \big] \Big| &\leq \left|\expec\left[\sum_{i=1}^{L} \Delta_i F_L\ind_{\{L\leq \theta\}}\right]\right| \nn\\
&\qquad +\left|\expec\left[ \sum_{i\neq j}^{L} (Y_i-X_i)(Y_j-X_j)f_{ij}^{(L)}\ind_{\{L\leq \theta\}}\right]\right|.
\end{align}
Since $\frac{\partial^2 F_\ell}{\partial x_i \partial x_j}$ is also uniformly bounded (\cite[Corollary~6.3]{DemMon10}), we obtain
\begin{align}
\left|\expec\left[\sum_{i\neq j}^{L} (Y_i-X_i)(Y_j-X_j)f_{ij}^{(L)}\ind_{\{L\leq \theta\}}\right]\right| &\leq \lambda_2 \expec\left[\sum_{i\neq j}^L |Y_i-X_i||Y_j-X_j|\ind_{\{L\leq\theta\}}\right] \nn\\
&\leq \lambda_2  \|X-Y\|_{\rm \sss MK}^2 \expec[L^2 \ind_{\{L\leq\theta\}}],
\end{align}
where $\lambda_2$ is allowed to change from line to line. The second moment of a non-negative integer-valued random variable $M$, can be written as
\eq
\expec[M^2]=\sum_{x=1}^\infty (2x - 1) \prob[M \geq x],
\en
so that
\begin{align}
\expec[L^2 \ind_{\{L\leq\theta\}}] &= \sum_{x=1}^{\infty} (2x-1) \prob[L \ind_{\{L\leq \theta\}}\geq x] = \sum_{x=1}^{\theta} (2x-1) \prob[L \ind_{\{L\leq \theta\}}\geq x]\nn\\
&\leq \sum_{x=1}^{\theta} 2x \prob[L\geq x] \leq \sum_{x=1}^{\theta} 2x \cdot c x^{-(\tau-1)} \leq \lambda_2 \theta^{-(\tau-3)}.
\end{align}
We split
\eq
\left|\expec\left[ \sum_{i=1}^L \Delta_i F_L \ind_{\{L\leq\theta\}}\right]\right| \leq \left|\expec\left[ \sum_{i=1}^L \Delta_i F_L \right]\right| + \left|\expec\left[ \sum_{i=1}^L \Delta_i F_L \ind_{\{L>\theta\}}\right]\right|.
\en
By symmetry of the functions $F_\ell$ with respect to their arguments, for i.i.d.\ $(X_i, Y_i)$ independent of $L$,
\eq
\expec\left[ \sum_{i=1}^L \Delta_i F_L \right] = \expec\left[L \Delta_1 F_L\right] = \expec\left[L (Y_1 - X_1) \int_{0}^{1} \frac{\partial F_L}{\partial x_1} (t Y_1+(1-t)X_1,X_2,\ldots,X_L) \dint t\right].
\en
Differentiating~\eqref{eq-defFl} gives, for $\ell\geq2$,
\eq
\frac{\partial}{\partial x_1}F_\ell(x_1,\ldots,x_\ell) = \psi(x_1, g_\ell(x_2,\ldots,x_\ell))-\frac{1}{\ell-1}\sum_{j=2}^{\ell}\psi(x_1,x_j),
\en
where $\psi(x,y)= xy / (1+\widehat{\beta}xy)$ and
\eq
g_\ell(x_2,\ldots,x_\ell) = \tanh\left(B+\sum_{j=2}^{\ell} \atanh(\widehat{\beta} x_j)\right),
\en
while differentiating~\eqref{eq-defF1} gives
\eq
\frac{\partial}{\partial x_1}F_\ell(x_1,x_2) = \psi(x_1, g_1)-\psi(x_1,x_2).
\en
Using that $\ell P_\ell = \overline{P} \rho_{\ell-1}$, we have that, with $K$ distributed as $\rho$,
\eq\label{eq-fixedpoint}
\expec[L \psi(X_1, g_L(X_2,\ldots,X_L))] = \overline{P}\expec[ \psi(X_1, g_{K+1}(X_2,\ldots,X_{K+1}))] = \overline{P}\expec[ \psi(X_1, X_2)],
\en
because $g_{K+1}(X_2,\ldots,X_{K+1})$ is a fixed point of~\eqref{eq-recursion}, so that $g_{K+1}(X_2,\ldots,X_{K+1}) \stackrel{d}{=}X_2$ and is independent of $X_1$. Therefore, one can show that
\eq \label{eq-ELdF0}
\expec[L \frac{\partial F_L}{\partial x_1}(x,X_2,\ldots,X_{\max\{2,L\}})] = 0, \qquad \text{for all }x\in[-1,1].
\en
Since $\frac{\partial F_L}{\partial x_1}$ is uniformly bounded, $L \frac{\partial F_L}{\partial x_1}$ is integrable, so that, by Fubini's theorem and~\eqref{eq-ELdF0},
\eq
\expec\left[ \sum_{i=1}^L \Delta_i F_L \right] =  \expec\left[ (Y_1 - X_1) \int_{0}^{1} \expec\left[L\frac{\partial F_L}{\partial x_1} (t Y_1+(1-t)X_1,X_2,\ldots,X_L)\Big|X_1,Y_1\right] \dint t\right] =0.
\en
Furthermore, by~\eqref{eq-deltaifl} and the uniform boundedness of $\frac{\partial F_\ell}{\partial x_i}$,
\eq
\left|\expec\left[ \sum_{i=1}^L \Delta_i F_L \ind_{\{L>\theta\}}\right]\right|\leq \expec\left[\sum_{i=1}^L |Y_i-X_i| c_1 \ind_{\{L>\theta\}}\right] \leq \lambda_1 \|X-Y\|_{\rm \sss MK}\theta^{-(\tau-2)}.
\en
Therefore, we conclude that
\eq
\Big| \expec\big[ (F_L (Y_1,\ldots,Y_L) - F_L(X_1,\ldots,X_L)) \ind_{\{L \leq \theta\}}\big] \Big| \leq  \lambda_1\|X-Y\|_{\rm \sss MK}\theta^{-(\tau-2)}
 + \lambda_2  \|X-Y\|_{\rm \sss MK}^2 \theta^{-(\tau-3)}. \label{eq-boundLsmall}
\en
Combining \eqref{eq-boundLbig} and \eqref{eq-boundLsmall} and letting $\theta = \|X-Y\|_{\rm \sss MK}^{-1}$ yields the desired result.
\end{proof}

\begin{lemma}[Dependence of $h^*$ on $\beta$]\label{lem-dephbeta}
Fix $B>0$ and $0 < \beta_1, \beta_2 \leq \beta_{\max}$. Let $h_{\beta_1}^*$ and $h_{\beta_2}^*$, where we made the dependence of $h^*$ on $\beta$ explicit, be the fixed points of \eqref{eq-recursion} for $(\beta_1,B)$ and $(\beta_2,B)$, respectively. Then, there exists a $\lambda < \infty$ such that
\eq
\|\tanh(h^*_{\beta_1})-\tanh(h^*_{\beta_2})\|_{\rm \sss MK} \leq \lambda |\beta_1-\beta_2|.
\en
\end{lemma}
\begin{proof}
For a given tree $\calT(\rho,\infty)$ we can, as in the proof of Proposition~\ref{prop-recursion}, couple
$\tanh(h^*_\beta)$ to the root magnetizations $m_\beta^{\ell,f/+}(B)$ such that, for all $\beta \geq 0$ and $\ell \geq 0$,
\eq
m_\beta^{\ell,f}(B) \leq \tanh(h^*_\beta) \leq m_\beta^{\ell,+}(B),
\en
where we made the dependence of $m^{\ell,f/+}$ on $\beta$ explicit. Without loss of generality, we assume that $0 < \beta_1 \leq \beta_2 \leq \beta_{\max}$. Then, by the GKS inequality,
\eq
|\tanh(h^*_{\beta_2})-\tanh(h^*_{\beta_1})| \leq m_{\beta_2}^{\ell,+}(B) - m_{\beta_1}^{\ell,f}(B) = m_{\beta_2}^{\ell,+}(B) - m_{\beta_2}^{\ell,f}(B) + m_{\beta_2}^{\ell,f}(B)- m_{\beta_1}^{\ell,f}(B).
\en
By Lemma~\ref{lem-45}, a.s.,
\eq
m_{\beta_2}^{\ell,+}(B) - m_{\beta_2}^{\ell,f}(B) \leq \frac{M}{\ell},
\en
for some $M < \infty$. Since $m_{\beta}^{\ell,f}(B)$ is non-decreasing in $\beta$ by the GKS inequality,
\eq
m_{\beta_2}^{\ell,f}(B)- m_{\beta_1}^{\ell,f}(B) \leq (\beta_2-\beta_1) \sup_{\beta_1 \leq \beta \leq \beta_{\max}} \frac{\partial m^{\ell,f}}{\partial \beta}.
\en
Letting $\ell \rightarrow \infty$, it thus suffices to show that $\partial m^{\ell,f}/ \partial \beta$ is, a.s., bounded uniformly in $\ell$ and $0<\beta_1 \leq \beta \leq \beta_{\max}$.

From \cite[\ch{Lemma~4.6}]{DemMon10} we know that
\eq
\frac{\partial }{\partial \beta}m^{\ell,f}(\beta,B) \leq \sum_{k=0}^{\ell-1} V_{k,\ell},
\en
with
\eq
V_{k,\ell} = \sum_{i \in \partial \calT(k)} \Delta_i \frac{\partial}{\partial B_i} m^{\ell}(\underline{B},0)\Big|_{\underline{B}=B}.
\en
By Lemma~\ref{lem-treepruning} and the GHS inequality,
\eq
\frac{\partial}{\partial B_i} m^{\ell}(\underline{B},0) = \frac{\partial}{\partial B_i} m^{\ell-1}(\underline{B},\underline{H}) \leq \frac{\partial}{\partial B_i} m^{\ell-1}(\underline{B},0),
\en
for some field $\underline{H}$, so that $V_{k,\ell}$ is non-increasing in $\ell$. We may assume that $B_i\geq B_{\min}$ for all $i \in \calT(\ell)$ for some $B_{\min}$. Thus, also using Lemma~\ref{lem-treepruning},
\begin{align}
V_{k,\ell} \leq V_{k,k+1}&=\sum_{i \in \partial \calT(k)}\Delta_i \frac{\partial}{\partial B_i} m^{k+1}(\underline{B},0)\Big|_{\underline{B}=B} \leq \sum_{i \in \partial \calT(k)}\Delta_i \frac{\partial}{\partial B_i} m^{k}(\underline{B},\xi\{\Delta_i\})\Big|_{\underline{B}=B} \nn\\
&= \frac{\partial}{\partial H} m^{k}(B,H \{\Delta_i\})\Big|_{H=\xi(\beta,B_{\min})},
\end{align}
where $\xi=\xi(\beta,B_{\min})$ is defined in~\eqref{eq-defxi}. By the GHS inequality this derivative is non-increasing in $H$, so that, by Lemma~\ref{lem-treepruning}, the above is at most
\eq
\frac{1}{\xi} \left[m^{k}(B, \xi \{\Delta_i\}) - m^{k}(B, 0)\right]
\leq \frac{1}{\xi} \left[m^{k+1}(B, 0) - m^{k}(B, 0)\right].
\en
Therefore,
\eq
\frac{\partial }{\partial \beta}m^{\ell,f}(\beta,B) \leq \sum_{k=0}^{\ell-1} V_{k,\ell} \leq \frac{1}{\xi} \sum_{k=0}^{\ell-1} \left[m^{k+1}(B, 0) - m^{k}(B, 0)\right] \leq \frac{1}{\xi}<\infty,
\en
for $0<\beta_1 \leq \beta \leq \beta_{\max}$.
\end{proof}

\begin{lemma}[Dependence of $h^*$ on $B$]\label{lem-dephB}
Fix $\beta \geq 0$ and $B_1, B_2 \geq B_{\min}>0$. Let $h_{B_1}^*$ and $h_{B_2}^*$, where we made the dependence of $h^*$ on $B$ explicit, be the fixed points of \eqref{eq-recursion} for $(\beta,B_1)$ and $(\beta,B_2)$, respectively. Then, there exists a $\lambda < \infty$ such that
\eq
\|\tanh(h^*_{B_1})-\tanh(h^*_{B_2})\|_{\rm \sss MK} \leq \lambda |B_1-B_2|.
\en
\end{lemma}
\begin{proof}
This lemma can be proved along the same lines as Lemma~\ref{lem-dephbeta}. Therefore, for a given tree $\calT(\rho,\infty)$, we can couple
$\tanh(h^*_B)$ to the root magnetizations $m^{\ell,f/+}(B)$ such that, for all $B > 0$ and $\ell \geq 0$,
\eq
m^{\ell,f}(B) \leq \tanh(h^*_B) \leq m^{\ell,+}(B).
\en
Without loss of generality, we assume that $0 < B_{\min} \leq B_1 \leq B_2$. Then, by the GKS inequality,
\eq
|\tanh(h^*_{B_2})-\tanh(h^*_{B_1})| \leq m^{\ell,+}(B_2) - m^{\ell,f}(B_1) = m^{\ell,+}(B_2) - m^{\ell,f}(B_2) + m^{\ell,f}(B_2)- m^{\ell,f}(B_1).
\en
By Lemma~\ref{lem-45}, a.s.,
\eq
m^{\ell,+}(B_2) - m^{\ell,f}(B_2) \leq \frac{M}{\ell},
\en
for some $M < \infty$. Since $m^{\ell,f}(B)$ is non-decreasing in $B$ by the GKS inequality,
\eq
m^{\ell,f}(B_2)- m^{\ell,f}(B_1) \leq (B_2-B_1) \sup_{B \geq B_{\min} > 0} \frac{\partial m^{\ell,f}}{\partial B}.
\en
Letting $\ell \rightarrow \infty$, it thus suffices to show that $\partial m^{\ell,f}/ \partial B$ is bounded uniformly in $\ell$ and $B \geq B_{\min} > 0$.
This follows from the GHS inequality:
\eq
\sup_{B \geq B_{\min} > 0} \frac{\partial m^{\ell,f}}{\partial B} \leq \left.\frac{\partial m^{\ell,f}}{\partial B}\right|_{B=B_{\min}} \leq \frac{2}{B_{\min}} \left[m^{\ell,f}(B_{\min})-m^{\ell,f}(B_{\min} / 2)\right] \leq \frac{2}{B_{\min}} < \infty.
\en
\end{proof}

\section{Convergence of the internal energy: proof of Proposition~\ref{prop-dpressure}}\label{sec-dpressure}

We shall start by identifying the thermodynamic limit of the intensive internal energy:
\begin{lemma}[From graphs to trees]\label{lem-dpsi}
Assume that the graph sequence $\{G_n\}_{n\geq1}$ is locally tree-like with asymptotic degree distribution $P$, where $P$ has finite mean, and is uniformly sparse. Then, a.s.,
\eq
\lim_{n\rightarrow\infty}\frac{\partial}{\partial \beta} \psi_n(\beta,B) = \frac{\overline{P}}{2} \expec\left[\big<\sigma_1\sigma_2\big>_{\nu'_2}\right],
\en
where $\nu'_2$ is defined in~\eqref{eq-defnu2}.
\end{lemma}
Lemma~\ref{lem-dpsi} shall be proved in Section~\ref{sec-dpsi}. Next, we will compute the derivative of $\varphi(\beta,B)$ with respect to $\beta$ in the following lemma and show that it equals the one on the graph:
\begin{lemma}[Tree analysis]\label{lem-dphi} Assume that distribution $P$ has strongly finite mean. Then,
\eq
\frac{\partial}{\partial \beta} \varphi(\beta,B) = \frac{\overline{P}}{2} \expec\left[\big<\sigma_1\sigma_2\big>_{\nu'_2}\right],
\en
where $\nu'_2$ is defined in~\eqref{eq-defnu2}.
\end{lemma}
Lemma~\ref{lem-dphi} shall be proved in Section~\ref{sec-dphi}. Lemmas~\ref{lem-dpsi} and~\ref{lem-dphi} clearly imply Proposition~\ref{prop-dpressure}.
\qed

\subsection{From graphs to trees: proof of Lemma~\ref{lem-dpsi}}\label{sec-dpsi}
This lemma can be proved as in \cite{DemMon10}. The idea is to note that
\eq \label{eq-rewritepsi}
\frac{\partial}{\partial \beta} \psi_n(\beta,B) = \frac{1}{n}\sum_{(i,j)\in E_n} \big<\sigma_i \sigma_j\big>_{\mu_n} =  \frac{|E_n|}{n}\frac{\sum_{(i,j)\in E_n} \big<\sigma_i \sigma_j\big>_{\mu_n}}{|E_n|}.
\en
By the local convergence and the uniform sparsity, we have that, a.s. (see \eqref{eq-onedges}),
\eq
\lim_{n\rightarrow\infty}\frac{|E_n|}{n} = \overline{P}/2.
\en
The second term of the right hand side of~\eqref{eq-rewritepsi} can be seen as the expectation with respect to a uniformly chosen edge $(i,j)$ of the correlation $\left<\sigma_i\sigma_j\right>_{\mu_n}$. For a uniformly chosen edge $(i,j)$, denote by $B_{(i,j)}(t)$ all vertices at distance from either vertex $i$ or $j$ at most $t$, and let $\partial B_{(i,j)}(t) = B_{(i,j)}(t) \setminus B_{(i,j)}(t-1)$. By the GKS inequality, for any $t\geq 1$,
\eq
\left<\sigma_i\sigma_j\right>^f_{B_{(i,j)}(t)} \leq \left<\sigma_i\sigma_j\right>_{\mu_n} \leq \left<\sigma_i\sigma_j\right>^+_{B_{(i,j)}(t)},
\en
where $\left<\sigma_i\sigma_j\right>^{+/f}_{B_{(i,j)}(t)}$ is the correlation in the Ising model on $B_{(i,j)}(t)$ with $+$/free boundary conditions on $\partial B_{(i,j)}(t)$.

Let $\overline{\mathcal{T}}(\rho,t)$ be the tree formed by joining the roots, $\phi_1$ and $\phi_2$, of two branching processes with $t$ generations and with offspring $\rho$ at each vertex, also at the roots. Then, taking $n\rightarrow\infty$, $B_{(i,j)}(t)$ converges to $\overline{\mathcal{T}}(\rho,t)$, because of the local convergence of the graph sequence. After all, a random edge can be chosen, by first picking a vertex with probability proportional to its degree, and then selecting a neighbor uniformly at random. Using this, one can show (see \cite[Lemma~6.4]{DemMon10}), also using the uniform sparsity, that, for all $t\geq1$, a.s.,
\eq
\lim_{n\rightarrow\infty} \expec_{(i,j)} \left[ \left<\sigma_i\sigma_j\right>^{+/f}_{B_{(i,j)}(t)} \right] = \expec\left[\left<\sigma_{\phi_1}\sigma_{\phi_2}\right>^{+/f}_{\overline{\mathcal{T}}(\rho,t)}\right],
\en
where the first expectation is with respect to a uniformly at random chosen edge $(i,j)\in E_n$ and the second expectation with respect to the tree $\overline{\mathcal{T}}(\rho,t)$. By Lemma~\ref{lem-treepruning} and Proposition~\ref{prop-recursion},
\eq
\lim_{t\rightarrow\infty}\expec\left[\left<\sigma_{\phi_1}\sigma_{\phi_2}\right>^{+/f}_{\overline{\mathcal{T}}(\rho,t)}\right] = \expec\left[\big<\sigma_1\sigma_2\big>_{\nu'_2}\right],
\en
thus proving the lemma.
\qed

\subsection{Tree analysis: proof of Lemma~\ref{lem-dphi}}\label{sec-dphi}
Let $X_i, i\geq 1$, be i.i.d.\ copies of $\tanh(h^*)$, also independent of $L$. Then, with $L$ having distribution $P$ and $F_\ell$ defined in~\eqref{eq-defFl} and~\eqref{eq-defF1},
\eq\label{eq-phiisF}
\varphi(\beta, B) = F_0 + \expec[F_L(X_1,\ldots,X_{\max\{2,L\}})],
\en
for some constant $F_0$ that is independent of $\beta$ and $B$.

From Proposition~\ref{prop-phiindeph} it follows that we can assume that $\beta$ is fixed in $h^*$ when differentiating $\varphi(\beta, B)$ with respect to $\beta$. Thus, taking the derivative of~\eqref{eq-phiisF} and using~\eqref{eq-fixedpoint}, one can show that
\eq
\frac{\partial}{\partial \beta} \varphi(\beta, B) = \frac{\overline{P}}{2} \widehat{\beta} + \frac{\overline{P}}{2} \expec \left[\psi(X_1, X_2)\right] =\frac{\overline{P}}{2} \expec \left[\widehat{\beta} + \frac{X_1 X_2}{1+\widehat{\beta} X_1 X_2}\right] = \frac{\overline{P}}{2} \expec \left[\frac{\widehat{\beta}+X_1 X_2}{1+\widehat{\beta} X_1 X_2}\right].
\en
Since, with $h_1,h_2$ i.i.d.\ copies of $h^*$,
\begin{align}
\expec\left[\frac{\widehat{\beta}+X_1 X_2}{1+\widehat{\beta} X_1 X_2}\right] &= \expec\left[\frac{\tanh(\beta)+\tanh(h_1)\tanh(h_2)}{1+\tanh(\beta) \tanh(h_1)\tanh(h_2)}\right] \nn\\
&= \expec\left[\frac{e^{\beta+h_1+h_2}-e^{-\beta-h_1+h_2}-e^{-\beta+h_1-h_2}+e^{\beta-h_1-h_2}}{e^{\beta+h_1+h_2}+e^{-\beta-h_1+h_2}+e^{-\beta+h_1-h_2}+e^{\beta-h_1-h_2}}\right] =\expec\left[\big<\sigma_1\sigma_2\big>_{\nu'_2}\right],
\end{align}
where $\nu'_2$ is given in \eqref{eq-defnu2}, we have proved the lemma.

\qed

\section{Thermodynamic quantities: proofs of Theorem~\ref{thm-thermqphi} and Corollary~\ref{cor-thermq}}\label{sec-thermq}

To prove the statements in Theorem~\ref{thm-thermqphi} we need to show that we can interchange the limit of $n\rightarrow\infty$ and the derivatives of the finite volume pressure. We can do this using the monotonicity properties of the Ising model and the following lemma:
\begin{lemma}[Interchanging limits and derivatives]\label{lem-interchange}
Let $\{f_n(x)\}_{n\geq1}$ be a sequence of functions that are twice differentiable in $x$. Assume that
\begin{itemize}
\item[(i)] $\lim_{n\rightarrow\infty} f_n(x)=f(x)$ for some function $x\mapsto f(x)$ which is differentiable in $x$;
\item[(ii)] $\frac{d}{dx} f_n(x)$ is monotone in $[x-h,x+h]$ for all $n\geq1$ and some $h>0$.
\end{itemize}
Then,
\eq
\lim_{n\rightarrow\infty} \frac{\dint}{\dint x}f_n(x) = \frac{\dint}{\dint x} f(x).
\en
\end{lemma}
\begin{proof}
First, suppose that $\frac{\dint^2}{\dint x'^2}f_n(x')\geq0$ for all $x'\in[x-h,x+h]$, all $n\geq1$ and some $h>0$. Then, for $h>0$ sufficiently small and all $n\geq1$,
\eq
\frac{f_n(x-h)-f_n(x)}{-h} \leq \frac{\dint}{\dint x}f_n(x) \leq \frac{f_n(x+h)-f_n(x)}{h},
\en
and taking $n\rightarrow\infty$ we get, by assumption~(i), that
\eq
\frac{f(x-h)-f(x)}{-h} \leq \liminf_{n\rightarrow\infty} \frac{\dint}{\dint x}f_n(x) \leq \limsup_{n\rightarrow\infty} \frac{\dint}{\dint x}f_n(x) \leq \frac{f(x+h)-f(x)}{h}.
\en
Taking $h\downarrow 0$ now proves the result. The proof for $\frac{\dint^2}{\dint x^2}f_n(x)\leq0$ is similar.
\end{proof}
We are now ready to prove Theorem~\ref{thm-thermqphi}.
\begin{proof}[Proof of Theorem~\ref{thm-thermqphi}]
We apply Lemma~\ref{lem-interchange} with the role of $f_n$ taken by $B\mapsto \psi_n(\beta,B)$, since
\eq
M_n(\beta,B)=\frac{1}{n} \sum_{i\in[n]} \left<\sigma_i\right>_{\mu_n}=\frac{\partial}{\partial B} \psi_n(\beta,B),
\en
and $\lim_{n\rightarrow\infty}\psi_n(\beta,B)=\varphi(\beta,B)$ by Theorem~\ref{thm-pressure} and $B\mapsto M_n(\beta,B)$ is non-decreasing by the GKS inequality. Therefore,
\eq
\lim_{n\rightarrow\infty}M_n(\beta,B)=\lim_{n\rightarrow\infty}\frac{\partial}{\partial B} \psi_n(\beta,B)=\frac{\partial}{\partial B} \varphi(\beta,B),
\en
which proves part~(a).

Part~(b) follows immediately from Proposition~\ref{prop-dpressure} and the observation that
\eq
U_n = -\frac{1}{n}\sum_{(i,j)\in E_n} \big<\sigma_i \sigma_j\big>_{\mu_n} = -\frac{\partial}{\partial \beta} \psi_n(\beta,B).
\en

Part~(c) is proved using Lemma~\ref{lem-interchange} by combining part~(a) of this theorem and that $B\mapsto \frac{\partial}{\partial B} M_n(\beta,B)$ is non-increasing by the GHS inequality.
\end{proof}

We can now prove each of the statements in Corollary~\ref{cor-thermq} by taking the proper derivative of $\varphi(\beta,B)$.
\begin{proof}[Proof of Corollary~\ref{cor-thermq}]
It follows from Theorem~\ref{thm-thermqphi}~(a) that the magnetization per vertex is given by
\eq
M(\beta,B) = \frac{\partial}{\partial B} \varphi(\beta, B).
\en
Similar to the proof of Lemma~\ref{lem-dphi}, we can ignore the dependence of $h^*$ on $B$ when differentiating $\varphi(\beta,B)$ with respect to $B$ by Proposition~\ref{prop-phiindeph}. Therefore, with $\hat{\beta}=\tanh(\beta)$,
\begin{align}
\frac{\partial}{\partial B} \varphi(\beta, B) &=\frac{\partial}{\partial B} \expec\left[\log\left(e^B \prod_{i=1}^{L} \{1+\tanh(\beta)\tanh(h_i)\} + e^{-B} \prod_{i=1}^{L} \{1-\tanh(\beta)\tanh(h_i)\} \right)\right] \nn\\
&= \expec\left[\frac{e^B \prod_{i=1}^L (1+\hat{\beta} \tanh(h_i)) - e^{-B} \prod_{i=1}^L (1-\hat{\beta} \tanh(h_i))}{e^B \prod_{i=1}^L (1+\hat{\beta} \tanh(h_i)) + e^{-B} \prod_{i=1}^L (1-\hat{\beta} \tanh(h_i))}\right] \nn\\
&= \expec\left[
\frac{e^B \prod_{i=1}^L \left(\frac{1+\hat{\beta} \tanh(h_i)}{1-\hat{\beta} \tanh(h_i)}\right)^{1/2}
- e^{-B} \prod_{i=1}^L \left(\frac{1-\hat{\beta} \tanh(h_i)}{1+\hat{\beta} \tanh(h_i)}\right)^{1/2}}
{e^B \prod_{i=1}^L \left(\frac{1+\hat{\beta} \tanh(h_i)}{1-\hat{\beta} \tanh(h_i)}\right)^{1/2}
+ e^{-B} \prod_{i=1}^L \left(\frac{1-\hat{\beta} \tanh(h_i)}{1+\hat{\beta} \tanh(h_i)}\right)^{1/2}}\right],
\end{align}
where $L$ has distribution $P$ and $\{h_i\}_{i\geq1}$'s are i.i.d.\ copies of $h^*$, independent of $L$. Using that $\atanh(x)=\frac{1}{2} \log\left(\frac{1+x}{1-x}\right)$ the above simplifies to
\eq
\expec\left[\frac{e^B \prod_{i=1}^L e^{\atanh(\hat{\beta} \tanh(h_i))} - e^{-B} \prod_{i=1}^L e^{-\atanh(\hat{\beta} \tanh(h_i))}}{e^B \prod_{i=1}^L e^{\atanh(\hat{\beta} \tanh(h_i))} + e^{-B} \prod_{i=1}^L e^{-\atanh(\hat{\beta} \tanh(h_i))}}\right]
=\expec\left[ \tanh\left(B + \sum_{i=1}^L \atanh(\hat{\beta} \tanh(h_i))\right)\right].
\en
By Lemma~\ref{lem-treepruning}, this indeed equals $\expec\left[\big<\sigma_0\big>_{\nu_{L+1}}\right]$, where $\nu_{L+1}$ is given in~\eqref{eq-defnuL1}, which proves part~(a).

Part~(b) immediately follows from Theorem~\ref{thm-thermqphi}(b) and Lemma~\ref{lem-dphi}.
\end{proof}

\paragraph*{Acknowledgements.}
The work of RvdH and SD is supported in part by The Netherlands Organisation for
Scientific Research (NWO). CG has been working at the Eindhoven University of Technology (TU/e) during the execution of this work; he acknowledges the support of the Department of Mathematics and Computer Science of the TU/e.

\end{document}